\documentclass[12pt]{article}

\usepackage{amsfonts,amsthm,amsmath,latexsym,amssymb}

\usepackage[dvips]{graphics}

\usepackage{color}

\usepackage{tikz} 
\usepackage{pgfplots}
\usetikzlibrary{decorations.markings}

\usepackage{amsthm}

\usepackage{hyperref}

\usepackage{epsf}
\usepackage{psfrag}

\pgfplotsset{soldot/.style={color=black,only marks,mark=*}}

\input epsf.tex
\newdimen\epsfxsize
\newdimen\epsfysize

\newcommand{\be}{\begin{equation}}
\newcommand{\ee}{\end{equation}}
\newcommand{\bes}{\begin{equation*}}
\newcommand{\ees}{\end{equation*}}

\renewcommand{\l}{\lambda}

\newcommand{\g}{\gamma}

\newtheorem{thm}{Theorem}[section]
\newtheorem{prop}[thm]{Proposition}

\newtheorem{cor}[thm]{Corollary}
\newtheorem{lemma}[thm]{Lemma}

\newtheorem{question}[thm]{Question}

\newtheorem*{thmtrace}{Proposition \ref{trace}}

\theoremstyle{definition}

\makeatletter

\@addtoreset{equation}{section} \makeatother

\def\1{{\bf 1}}

\begin{document}


\title{\bf Tangential Loewner hulls}
\bigskip
\author{{\bf Joan Lind}
\\
\\
}

\maketitle

\abstract{
Through the Loewner equation, real-valued driving functions generate sets called Loewner hulls.  We analyze driving functions that approach 0 at least as fast as $a (T-t)^r$ as $t \to T$, where $r \in (0, 1/2)$, and show that the corresponding Loewner hulls have tangential behavior at time $T$.
We also prove a result about trace existence and apply it to show that the Loewner hulls driven by $a(T-t)^r$ for $r \in (0,1/2)$ have a tangential trace curve.
}
\vspace{0.1in}




\tableofcontents

\bigskip

\section{Introduction and results}\label{intro}

The Loewner equation provides a correspondence between continuous functions (called driving functions) and certain families of growing sets (called hulls).  
We are interested in the question of how analytic properties of the driving functions affect geometric properties of the hulls, a question that has inspired much research (such as  \cite{MR}, \cite{L}, \cite{LMR}, \cite{W}, \cite{LT}, \cite{KLS}, \cite{ZZ}, among others.)

In this paper, we examine the end behavior of Loewner hulls driven by functions that are bounded below by $a (T-t)^r$, where $r \in (0, 1/2)$.  
We show that this results in tangential hull behavior at the end time (noting that by scaling, we may simply take $T=1$).

\begin{thm}\label{tangentialapproach}
Assume that $\l$ is a driving function defined on $[0,1]$ 
satisfying that  $\l(1) = 0$ and $\l(t) \geq a (1-t)^r$ for $a\geq 4$ and $r \in (0,1/2)$.  
Let $K_t$ be the Loewner hull generated by $\l$, and let $p=\inf \{ x \in K_1 \cap \mathbb{R} \}$.
Then near $p$, $K_1$ is contained in the region  $ \{x+iy \, : \, 0 \leq x, \, 0 \leq y \leq C (x-p)^{2-2r}\}$
for $C =C(a,r) > 0$.
\end{thm}

This is the counterpoint to a result in \cite{KLS} which analyzes the initial behavior of hulls driven by functions that begin faster than $at^r$ for $r \in (0,1/2)$ and shows that
  these hulls leave the real line tangentially.  
  The end-hull question, however, is slightly harder to analyze due to the influence of the past on hull growth.

We view Theorem \ref{tangentialapproach} as a partial extension of the following result from \cite{LMR} to the  $\kappa=\infty$ case.

\begin{thm}[Theorem 1.3 in \cite{LMR}]\label{LMRthm}
If $\l:[0,T] \to \mathbb{R}$ is sufficiently regular on $[0,T)$ and if 
$$\lim_{t \to T} \frac{|\l(T) - \l(t)|}{\sqrt{T-t}} = \kappa >4,$$
then the trace $\g$ driven by $\l$ satisfies that $\displaystyle \gamma(T) = \lim_{t \to T} \gamma(t)$
exists, is real, and $\g$ intersects $\mathbb{R}$ in the same angle as the trace for $\kappa\sqrt{1-t}$.
\end{thm}

Theorem \ref{tangentialapproach} addresses the approach to $\mathbb{R}$, 
but it does not address the question of the existence of a trace.
To give a fuller extension, we address the existence of the trace in the following result.

\begin{prop}\label{trace}
If $\l:[0,T] \to \mathbb{R}$ is sufficiently nice on $[0,T)$ 
with $ |\lambda(T)-\lambda(t)| \geq 4\sqrt{T-t}$ for all $t \in [0,T)$,
then the trace $\g$ driven by $\l$ satisfies that $\displaystyle \gamma(T) = \lim_{t \to T} \gamma(t)$
exists and is real.
\end{prop}

The needed assumption of Proposition \ref{trace}, which utilizes the notion of Loewner curvature introduced in \cite{LRcurv}, will be made explicit in Section \ref{Tr}.  
Taken together, Theorem \ref{tangentialapproach} and Proposition \ref{trace} provide an understanding of the hulls driven by functions $a(T-t)^r$, as illustrated in Figure \ref{thirdroot}.

\begin{figure}
\centering
\includegraphics[scale=.45]{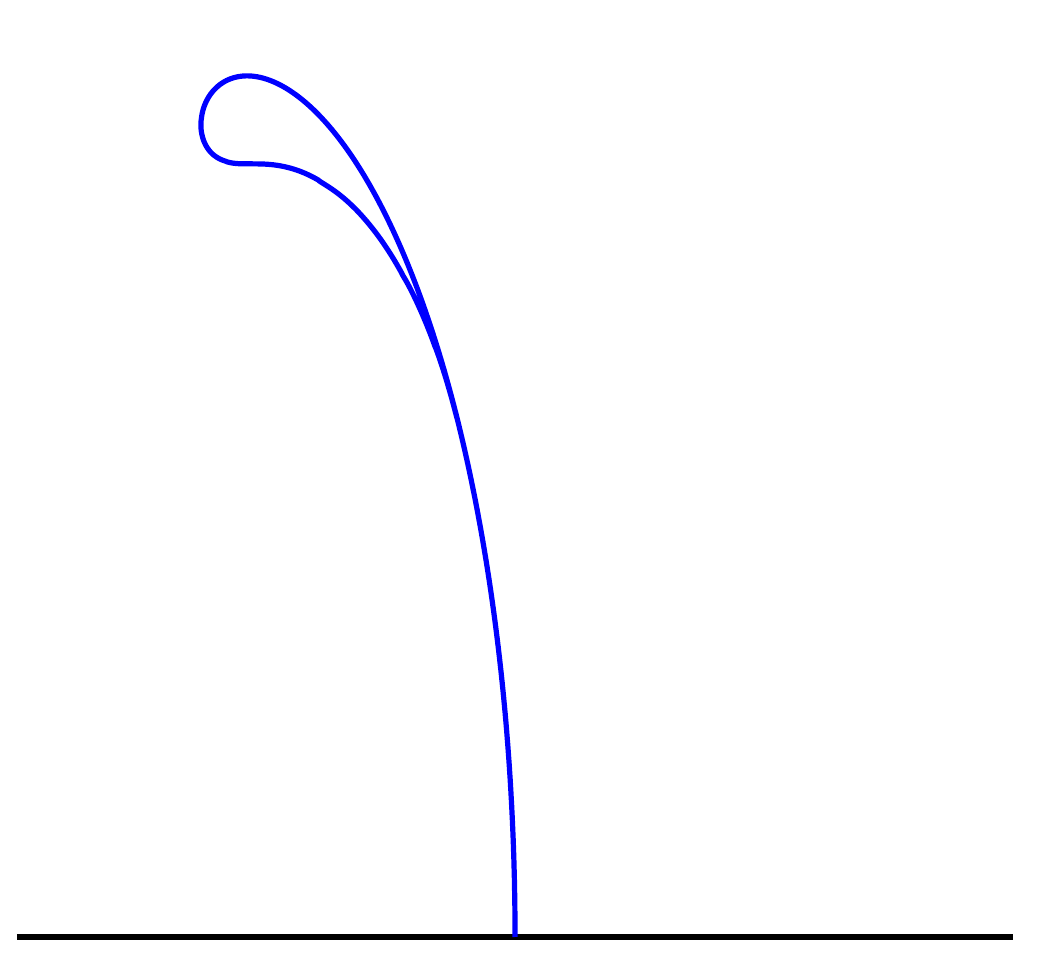} \hspace{0.4in}
\includegraphics[scale=.45]{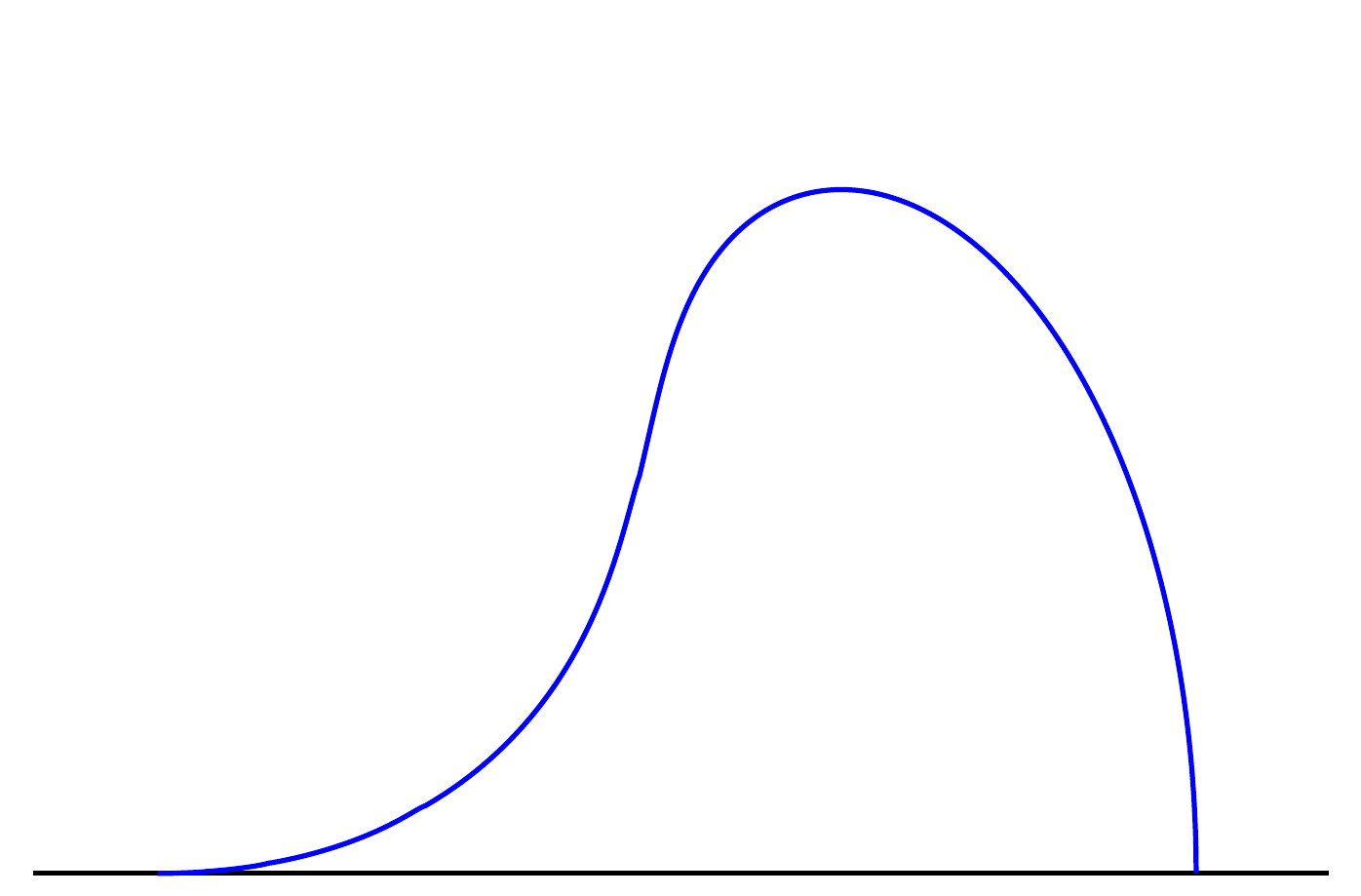}
\caption{The trace curve driven by $a(1-t)^{1/3}$ hits back on itself tangentially when  $a=2.5$ (left) and hits $\mathbb{R}$ tangentially when $a=4$ (right).
} \label{thirdroot}
\end{figure}

\begin{cor}\label{rdrivers}
Let $a\neq 0$ and $r \in (0, 1/2)$.  The Loewner hulls generated by $\lambda(t) = a(T-t)^r$ have a trace curve for $t \in [0,T]$.  This curve approaches the real line or itself tangentially as $t \to T$.
\end{cor}

We have interest in applying Theorem \ref{tangentialapproach} to some driving functions that lack the regularity of $a(T-t)^r$.  See Figure \ref{Wthirdroot} for one such example.  We will briefly discuss this and other examples in the last section.

\begin{figure}
\centering
\includegraphics[scale=.8]{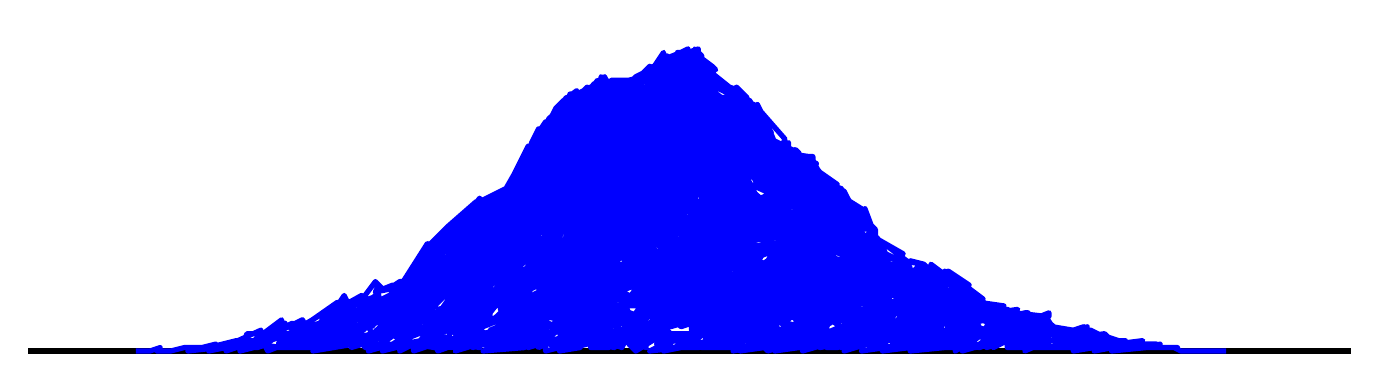} 
\caption{The Loewner hull $K_{\pi}$ driven by Weierstrass function $\displaystyle 4 \sum_{k=0}^\infty 3^{-n/3} \cos(3^n t)$.
} \label{Wthirdroot}
\end{figure}

  Due to our desire to understand the hulls of less regular driving functions, one might ask if there are weaker conditions than those of Proposition \ref{trace} that would still give the existence of a trace.  In general, the question of the existence of the trace is difficult and there has not been much progress on this front (as a notable exception to this statement, see the work in \cite{ZZ}).  We further discuss this question in the last section, and we give an example to show that monotonicity, while used in the proof of Proposition \ref{trace}, is not enough to guarantee the trace existence.

We end with a brief note about the organization of this paper.  Section \ref{LE} contains  background on the Loewner equation, and Section \ref{Pf} contains the proof of  Theorem \ref{tangentialapproach}.  In Section \ref{Tr} we explore the trace existence question by proving Proposition \ref{trace} and Corollary \ref{rdrivers} and discussing some examples.

\section{Loewner equation background}\label{LE}  

This section briefly introduces the relevant background regarding the Loewner equation.  
See \cite{lawler} for a more detailed introduction.

We work with the chordal Loewner equation in the setting of the upper halfplane $\mathbb{H}$.  In this context, the Loewner equation is the following initial value problem:
\begin{equation}\label{downLE}
\partial_t g_t(z) = \frac{2}{g_t(z) - \lambda(t)}, \;\;\;\;\; g_0(z) = z
\end{equation}
where $\lambda$ is a continuous real-valued function 
and $z \in \overline{\mathbb{H}} $.
For each initial value $z\in \overline{\mathbb{H}} \setminus \{ \l(0) \}$, a unique solution to \eqref{downLE} exists as long as the denominator remains non-zero.
  We collect the initial values that lead to a zero in the denominator into sets called hulls:
$$ K_t = \{ z \, : \, g_s(z) = \l(s) \text{ for some } s \in [0,t] \}.$$
One can show that $\mathbb{H} \setminus K_t$ is simply connected and $g_t$ is a conformal map from $\mathbb{H} \setminus K_t$ onto $\mathbb{H}$.  
Since the driving function $\l$ determines the families of hulls $K_t$, we say that $\l$ generates $K_t$ or that $K_t$ is driven by $\l$.

In many cases, there is a curve $\gamma$ (called a trace) so that $K_t$ is the complement of the unbounded component of $\mathbb{H} \setminus \gamma[0,t]$ for all $t$.  
 When the trace $\gamma$ is a simple curve in $\mathbb{H} \cup \{ \l(0) \}$, then the situation is especially nice and we have that $K_t = \gamma[0,t]$.
In this case, $g_t$ can be extended to the tip $ \g(t) $ and $g_t( \g(t)) = \l(t)$.

\begin{figure}
\centering
\includegraphics[scale=.45]{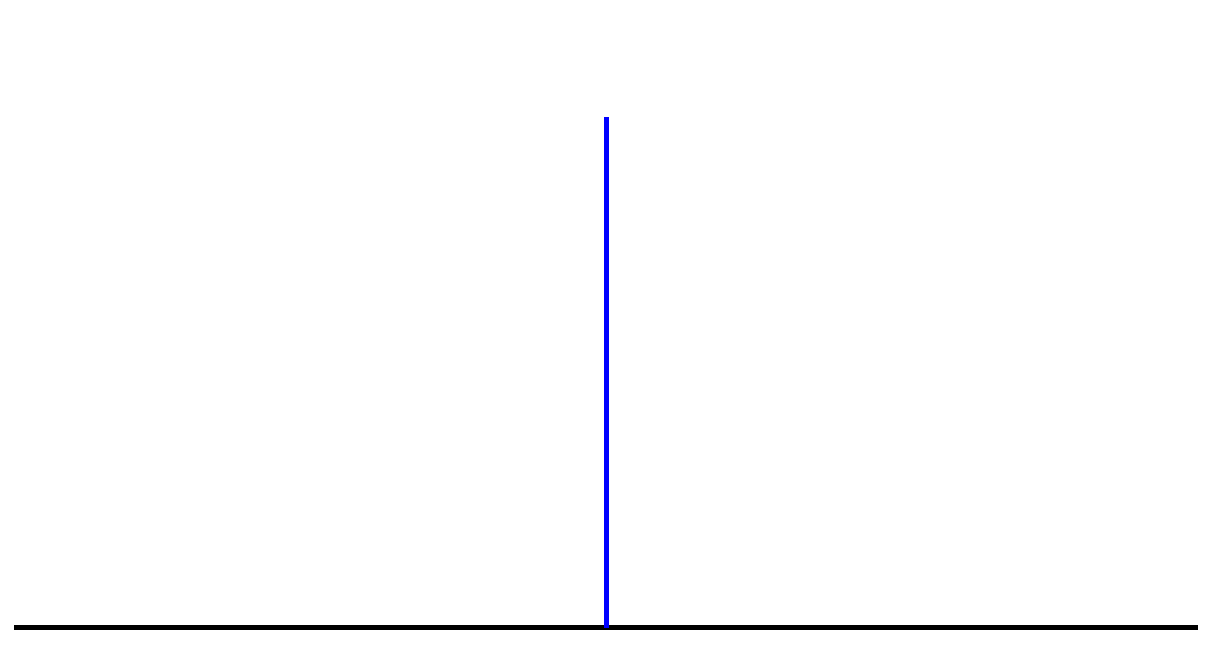} \hspace{0.4in}
\includegraphics[scale=.42]{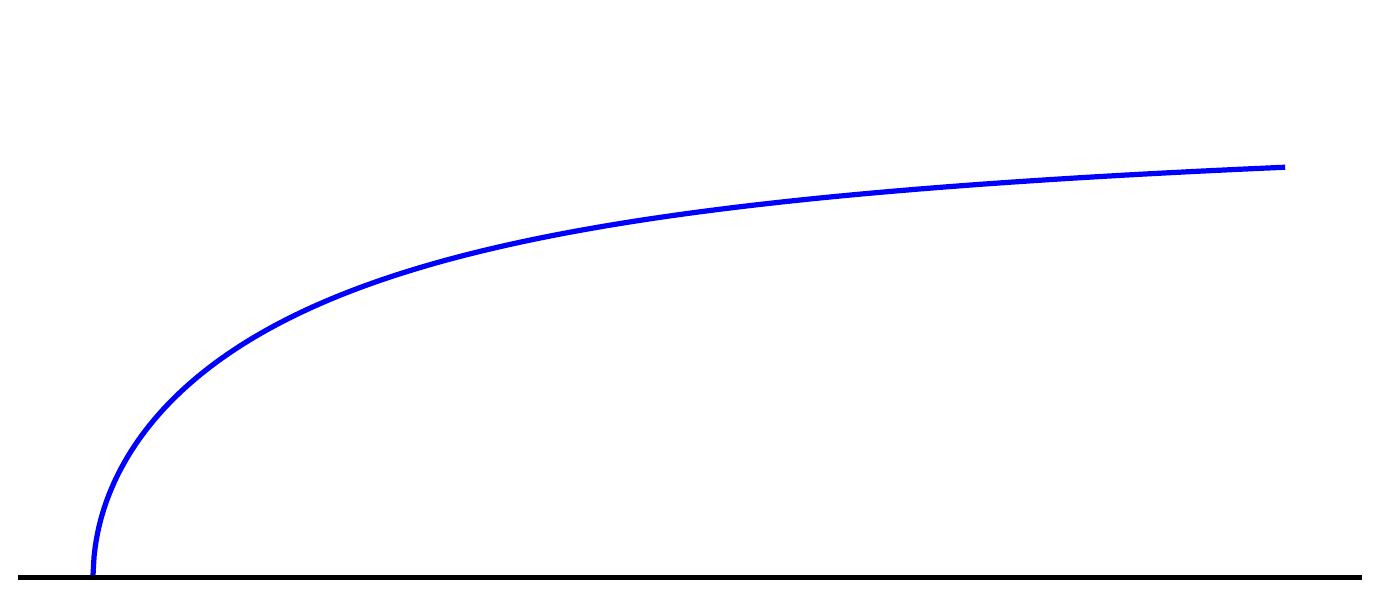}

\vspace{0.1in}

\includegraphics[scale=.42]{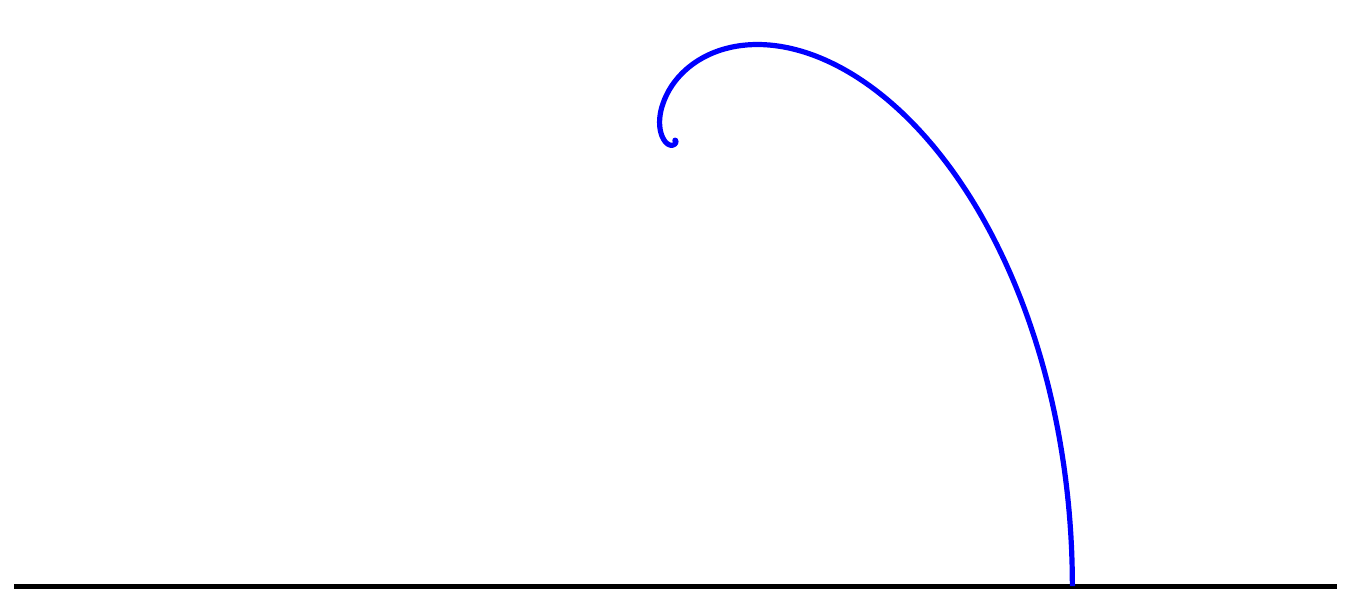} \hspace{0.4in}
\includegraphics[scale=.42]{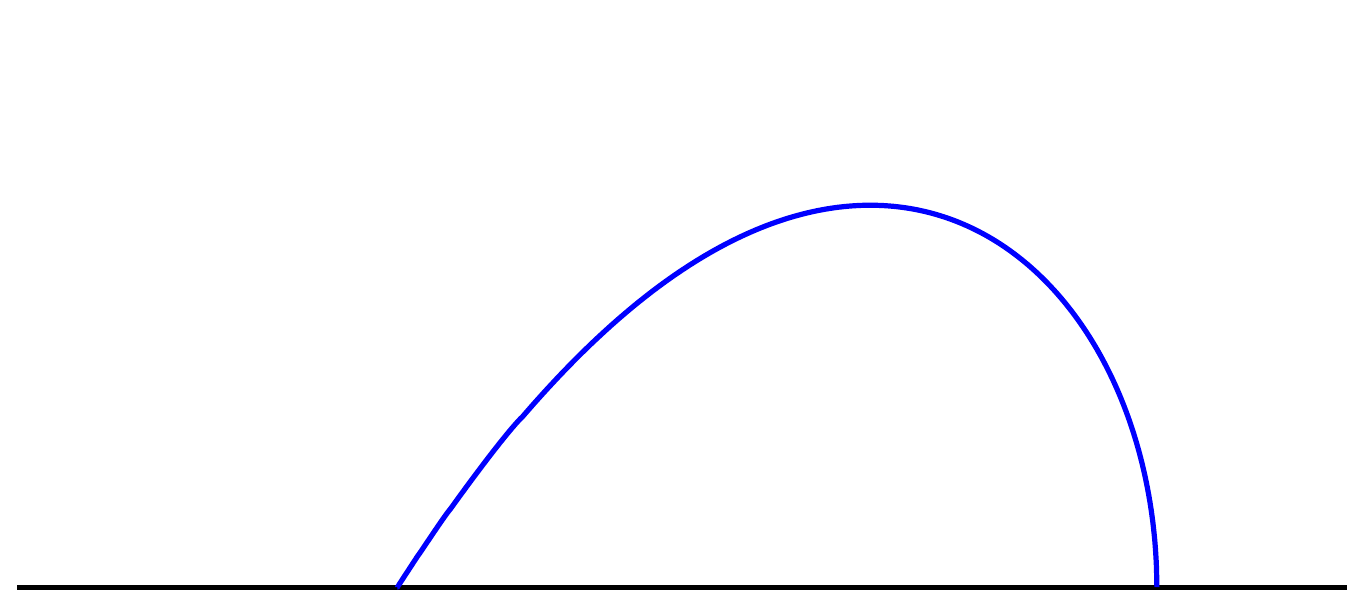}
\caption{Top left: The hull $K_t= [a, a+i2\sqrt{t}]$ driven by the constant driving function $\l \equiv a$. 
Top right: The hull  driven by a linear driving function $\l(t) = at$ with $a>0$.
Bottom left: The hull $K_1$ driven by $3\sqrt{1-t}$.
Bottom right:  The hull $K_1$ driven by $5\sqrt{1-t}$ contains the blue trace curve $\g$
  and the points under $\g$ in $\overline{\mathbb{H}}$.
} \label{Examples}
\end{figure}

See Figure \ref{Examples} for some example Loewner hulls, which were  computed in \cite{KNK}.  Note that in the bottom right example, the hull $K_1$ driven by $5\sqrt{1-t}$ has a trace $\g$ which is not a simple curve in $\mathbb{H} \cup \{ \l(0) \}$.  As a result, the hull also contains the points under the curve and a real interval.
The third and fourth hulls in Figure \ref{Examples} are from an important family of driving functions 
$k\sqrt{1-t}$.  We will use the following useful information about this family:

\begin{lemma} \label{capture}
Let $k \geq 4$. The Loewner hull $K_1$ driven by $k\sqrt{1-t}$ contains the real points in the
 interval $[(k-\sqrt{k^2-16})/2, \, k].$
Further, the Loewner hull $K_1$ driven by $\l$ with $\l(1)=0$ and $\l(t) \geq k\sqrt{1-t}$ contains the real interval $[(k-\sqrt{k^2-16})/2, \l(0)]$.
 \end{lemma}
 
The first statment can be established by a short computation (see, for instance, Lemma 2.3 in \cite{LR}).
The second statement follows from a comparison between the driving functions $\l$ and  $k\sqrt{1-t}$ (see, for instance, the proof of Theorem 1.1 in \cite{LR}).

While Lemma \ref{capture} can be used to determine when a Loewner hull is not a simple curve, the following theorem gives a large class of hulls that are simple curves. 
We use the notation
$$||\lambda||_{1/2} = \sup_{t \neq s} \frac{\l(t) - \l(s)}{\sqrt{|t-s|}}.$$

\begin{thm}[Theorem 2 in \cite{L}]\label{cis4}
If  $||\lambda||_{1/2} < 4$, then the hulls $K_t$ driven by $\lambda$ satisfy $K_t = \g[0,t]$ for a simple curve $\g$ contained in $\mathbb{H} \cup \{\lambda(0) \}$.
\end{thm}

Additional driving function regularity provides additional regularity for the associated trace curves.

\begin{thm}[\cite{W}, \cite{LT}]\label{smoothtrace}
Let $\l \in C^\beta[0,T]$ for $\beta >1/2$ with $\beta + 1/2 \notin \mathbb{N}$.
Then the Loewner trace driven by $\l$ is in $C^{\beta+ \frac{1}{2}}(0,T]$.
\end{thm}

Loewner hulls satisfy  some useful properties, which we will utilize frequently.
If a driving function  $\l$ generates hulls  $K_t$, then
 the following hold:
\begin{itemize}
\item {\bf Translation:}  For $a \in \mathbb{R}$, the driving function $\l(t) + a$ generates hulls $K_t + a$.
\item {\bf Scaling:} For $k >0$, the driving function $k\l(t/k^2)$ generates hulls $k K_{t/k^2}$. 
\item {\bf Reflection:} The driving function $-\l(t)$ generates hulls $R_I(K_t)$, where $R_I$ denotes reflection about the imaginary axis.
\item {\bf Concatenation:} For $s \in (0,T)$, the driving function $\l(s+t)$ generates hulls $g_s(K_{s+t})$.
\end{itemize}

There is an alternate flow that one can use to generate Loewner hulls.  Setting $\xi(t) = \l(s-t)$ for $t \in [0,s]$, let $f_t$ satisfy the following initial value problem:
\begin{equation}\label{upLE}
\partial_t f_t(z) = \frac{-2}{f_t(z) - \xi(t)}, \;\;\;\;\; f_0(z) = z.
\end{equation}
Then $f_s = g_s^{-1}$ (where $g_t$ is the solution to \eqref{downLE} driven by $\l$), and so the hull $K_s$ driven by $\l$ is the closure of $\mathbb{H} \setminus f_s(\mathbb{H})$.
We refer to \eqref{upLE} as the upward Loewner flow, since  $\partial_t \text{Im}(f_t(z)) > 0$ for $z \in \mathbb{H}$.

For the convenience of the reader, we end this section with statements of results from other papers (possibly rewritten in our notation) that we will use.

\begin{lemma}[Lemma 3.3b in \cite{CR}] \label{CRlem}
Let $0<\epsilon<1$. If $I\subset\mathbb{R}$ is an interval of length $\sqrt{T}$ and $10I$ the concentric interval of size $10\sqrt{T}$, and if
$    \int_{0}^{T}\mathbf{1}_{\{\lambda(t)\in 10I\}}\mathrm{d}t\leq\epsilon T,$
then
$    K_T\cap I\times[4\sqrt{\epsilon T},\infty)=\emptyset.$ 
\end{lemma}

\begin{lemma}[Lemma 4.2 in \cite{ZZ}]\label{ZZlem}
Let $I= \{x \in \mathbb{R} \cap K_1 \setminus \cup_{t<1} K_t \}$.  If $I$ is an interval and there exists $x_0 \in I^\circ$ and $c >0$ so that
$$ \frac{|\l(t) - g_t(x_0)|}{\sqrt{1-t}} > c$$
for all $t \in [0,1)$,
then there exists an open set $B$ in $\mathbb{C}$ containing $I^\circ$ so that
$B \cap \overline{\mathbb{H}} \subset K_1 \setminus \cup_{t<1} K_t$.
\end{lemma}

The last result uses the concept of Loewner curvature introduced in \cite{LRcurv}.  
For a driving function $\l \in C^2[0,T)$ the Loewner curvature  can be computed by
\begin{equation}\label{LCdef}
LC_{\l}(t) =  
    \begin{cases}
       0 &\quad\text{if } \l'(t) = 0\\
         \frac{\l'(t)^3}{\l''(t)} &\quad\text{otherwise} \\ 
     \end{cases}.
\end{equation}
Note that driving functions $\alpha + c\sqrt{\tau-t}$ have constant Loewner curvature $c^2/2$.
The Loewner curvature comparison principle (which is stated below in part) allows for comparison with the hulls generated by constant curvature driving functions.

\begin{thm}[Theorem 15b in \cite{LRcurv}]\label{LCthm}
Let $\gamma$ be the trace driven by $\l \in C^2[0,T)$. 
If $ 9 \leq c^2/2 \leq LC_{\l}(t) < \infty$, then 
$\gamma[0,T)$ does not intersect the interior of the hull $K^*_\tau$ driven by
 $\mu(t) =  \alpha + c\sqrt{\tau-t}$, where $\alpha$ and $\tau$ are chosen so that $\l(0) = \mu(0)$ and $\l'(0) = \mu'(0)$. 
\end{thm}

\section{Proof of the tangential result}\label{Pf}

In this section, we prove Theorem \ref{tangentialapproach}.  Our first step is to consider the mapped down hull $\hat{K}_{s,1} := g_s(K_1 \setminus K_s)$ and 
show that this hull must be low near 0 (see Lemma \ref{step1}).
Then, in the second step, we
watch points from $\hat{K}_{s,1}$ under the upward Loewner flow to gain bounds on 
$K_1$ near $p$.

\begin{lemma}\label{fromCRlemB}
Suppose $\l$  is defined on $[0,1]$ and satisfies that 
$\lambda(1)=0$ and $\lambda(t) \geq k \sqrt{1-t}$ for $t \in [0,1]$.
Let $K_t$ be driven by $\l$.
Then $\displaystyle K_1 \cap \left( (-\infty,2] \times [26/k, \infty) \right) = \emptyset$.
\end{lemma}

\begin{proof}
When $k < 13$, the result is trivially true because any Loewner hull at time $t=1$ has its height bounded by 2, and so we assume $k \geq 13$.  Let $I \subset (-\infty, 2]$ be an interval of length 1.  
The amount of time that $\l$ spends  in $10I$, the concentric interval of length 10 which is contained in $(-\infty, 6.5]$, is at most  $(6.5/k)^2$.  
Therefore, applying Lemma \ref{CRlem} with $\epsilon = (6.5/k)^2$, 
we conclude that $K_1$ does not intersect $I \times [26/k, \infty)$.
\end{proof}

\begin{lemma}\label{step1}
Suppose that $\l$  is defined on $[0,1]$ and satisfies that $\l(1)=0$ and $\l(t) \geq a (1-t)^r$  where $a\geq 4$ and $r \in (0, 1/2)$.  
Then for $s<1$, $\hat{K}_{s,1}= g_s(K_1 \setminus K_s)$ satisfies that 
\begin{equation}\label{low}
\displaystyle \hat{K}_{s,1} \cap \left( (-\infty,2\sqrt{1-s}] \times [26 a^{-1}(1-s)^{1-r}, \infty) \right) = \emptyset.
\end{equation}
Further, $\inf \{x \in \hat{K}_{s,1} \cap \mathbb{R} \} \leq \frac{8}{a}(1-s)^{1-r}$.
\end{lemma}

\begin{proof}
The rescaled hull $\frac{1}{\sqrt{1-s}} \hat{K}_{s,1}$ is generated by the driving function
$$\hat{\l}(t) = \frac{1}{\sqrt{1-s}}  \l(s+t(1-s)), \;\;\;   t \in [0,1]$$ 
which satisfies that $\hat{\l}(1)=0$ and $\hat{\l}(t) \geq a (1-s)^{r-1/2} \sqrt{1-t}$.
To obtain \eqref{low}, we apply Lemma \ref{fromCRlemB} with $k = a (1-s)^{r-1/2}$ 
and then
rescale by $\sqrt{1-s}$.

To establish the second statement, 
we apply Lemma \ref{capture} to driving function $\hat{\l}$.  
Thus for $k = a (1-s)^{r-1/2}$, the hull driven by $\hat{\l}$ contains the point
$$\frac{k-\sqrt{k^2-16}}{2} = \frac{8}{k+\sqrt{k^2-16}} \leq \frac{8}{k}.$$
In other words, there is a real point $\hat{p}$ in the hull  
$\frac{1}{\sqrt{1-s}} \hat{K}_{s,1}$ with $\hat{p} \leq \frac{8}{a}(1-s)^{1/2-r}$.  Scaling then gives the desired result.
\end{proof}

Next we need to analyze the upward Loewner flow.  
For $s$ fixed  and $t \in [0,s]$, we set $\xi_t= \l(s-t)$ and let $f_t = x_t + i y_t$ satisfy \eqref{upLE},
which can be decomposed into the pair of equations
\begin{equation} \label{upLExy}
  \partial_t x_t = 2 \frac{\xi_t - x_t}{(\xi_t-x_t)^2 + y_t^2} \;\;\;\;  \text{and} \;\;\;\;
      \partial_t y_t = 2 \frac{ y_t}{(\xi_t-x_t)^2 + y_t^2}.
\end{equation}

\begin{lemma}\label{UpwardBounds}
Let $s \in (0,1)$ be fixed and let $\xi_t$ satisfy $\xi_t \geq a(1-s + t)^r$ for $t\in[0,s]$ and $r \in (0,1/2)$.
Let $x_t$ and $y_t$ be the solutions to \eqref{upLExy} with initial values $x_0$ and $y_0$, respectively, and let $p_t=f_t(p_0^s)$ be the solution  to \eqref{upLE} with initial value $p^s_0 = \inf \{ x \in \hat{K}_{s,1}\cap \mathbb{R} \}$.
There exists $M=M(r) \geq 4$, so that when $a \geq M$, the following hold:
\begin{enumerate}
\item[(i)] If $x_0 \leq 2\sqrt{1-s}$, then $x_t \leq 2\sqrt{1-s+t}$ for all $t \in [0,s]$.
\item[(ii)] If $x_0 \leq 2\sqrt{1-s}$, then $y_t \leq 2 y_0$ for all $t \in [0,s]$.
\item[(iii)] If  $x_0 \in [p^s_0+\sqrt{1-s}, \, 2\sqrt{1-s}]$ and $y_0<  26 a^{-1} (1-s)^{1-r}$, then
   $x_t-p_t \geq x_0-p^s_0$ for all $t \in [0,s]$. 
\end{enumerate}
\end{lemma}

\begin{proof}
Assume $x_0 \leq 2\sqrt{1-s}$.  Let $\tau \in [0,s]$ be a time when  $ x_\tau = 2\sqrt{1-s+\tau}$.  
Then since $\xi_\tau \geq a(1-s+\tau)^r \geq a\sqrt{1-s+\tau}$,
$$ \left. \left( \partial_t \, x_t \right) \right\vert_{t=\tau} \leq  \frac{2}{\xi_\tau-x_\tau } 
    \leq \frac{2}{(a-2)\sqrt{1-s+\tau}} \leq \frac{1}{\sqrt{1-s+\tau}} 
    = \left.\left( \partial_t \, 2\sqrt{1-s+t}  \right)\right\vert_{t=\tau}.$$
This implies that $x_t$ can never surpass $2\sqrt{1-t+s}$ and hence (i) holds.

For (ii), we continue to assume that $x_0 \leq 2\sqrt{1-s}$. Then by (i), we have that 
$$\xi_t-x_t \geq a(1-s+t)^r - 2\sqrt{1-s+t} \geq (a-2)(1-s+t)^r.$$
At times $\tau$ when $y_\tau \leq 2 y_0$, we have that
$$ \left.\left( \partial_t \, y_t  \right)\right\vert_{t=\tau} \leq \frac{4y_0}{(a-2)^2(1-s+\tau)^{2r}}
   \leq  \left. \partial_t \left(y_0+\frac{4y_0}{(a-2)^2(1-2r)}(1-s+t)^{1-2r} \right)\right\vert_{t=\tau}$$
We choose $a$ large enough so that $4(a-2)^{-2}(1-2r)^{-1} \leq 1$.
Thus we can conclude that $y_t$ remains bounded by $y_0+y_0(1-s+t)^{1-2r} \leq 2y_0$.

Lastly, assume that  $x_0 \in [p^s_0+\sqrt{1-s}, \, 2\sqrt{1-s}]$ and $y_0<  26 a^{-1} (1-s)^{1-r}$, 
and assume that $a$ is large enough for (ii) to hold.  
Then 
\begin{equation}\label{xminusp}
\partial_t\left( x_t - p_t \right) = 2 \frac{\xi_t - x_t}{(\xi_t-x_t)^2 + y_t^2} - \frac{2}{ \xi_t -p_t} 
      = 2 \frac{(\xi_t - x_t)(x_t - p_t) - y_t^2}{(\xi_t-p_t)[(\xi_t-x_t)^2+y_t^2]}.
\end{equation}
Let $\tau \in [0,s]$ be a time when  $ x_\tau -p_\tau = x_0 - p_0^s$.      
 Our goal is to show that the numerator $(\xi_\tau - x_\tau)(x_\tau - p_\tau) - y_\tau^2 >0$, meaning that $x_t - p_t$ is increasing at time $\tau$ (and hence $x_t - p_t \geq x_0 - p_0^s$ for all $t \in [0,s])$.
Now by applying (i) and (ii) and the fact that $x_0 - p_0^s \geq \sqrt{1-s}$, we obtain
\begin{align*}
 (\xi_\tau - x_\tau)(x_\tau - p_\tau) - y_\tau^2  
          &\geq (a-2)\sqrt{1-s+\tau}\sqrt{1-s} - \left(\frac{52}{a}\right)^2(1-s)^{2-2r}\\
          &\geq (1-s)\left[ a-2 - \left(\frac{52}{a}\right)^2(1-s)^{1-2r} \right]
\end{align*}
We can guarantee that this is positive by taking $a\geq 15$.
\end{proof}

\begin{proof}[Proof of Theorem \ref{tangentialapproach}]

We will first prove the case when $a \geq M$ (i.e. $a$ is large enough for Lemma \ref{UpwardBounds} to hold).
Set $K^* := K_1 \cap \{z \, : \, \text{Re}(z) \leq 2 \} $.  
We will show that $K^*$ is contained in the region $ \{x+iy \, : \, 0 \leq x, \, 0 \leq y \leq 52a^{-1}(x-p)^{2-2r}\}$.

Let $s \in [0,1)$ and let $z \in \hat{K}_{s,1}$ with Re$(z) \in  [p^{s}_0+\sqrt{1-s}, \, 2\sqrt{1-s}]\}$.
Then  Lemma \ref{step1} and Lemma \ref{UpwardBounds} imply that
$$\text{Im}\left( f_{s}(z) \right) \leq 2 \,\text{Im}(z) \leq 52 a^{-1}(1-s)^{1-r}$$
and
$$\text{Re}\left( f_{s}(z) \right) - p \geq \text{Re}(z) - p^{s}_0,$$
where we used that $f_s(p^s_0) = p$.
Thus 
$$ f_{s}(z) \in [p+\sqrt{1-s}, \, \infty] \times [0,\, 52 a^{-1} (1-s)^{1-r}].$$

It remains to show that 
\begin{equation}\label{K}
K^* \subset \bigcup_{s \in [0, 1)} f_s\left( \hat{K}_{s,1} \cap \{z \, : \, \text{Re}(z) \in [p_0^s+\sqrt{1-s}, \, 2\sqrt{1-s}] \} \right),
\end{equation}
which will follow once we show the boundary $\partial K_1$ in $\{x+iy \, : \, x \leq 2, y >0\}$ is contained in the right hand side of \eqref{K}.
Let $z \in \partial K_1 \cap \{x+iy \, : \, x \leq 2, y >0\}$.  Then $z$ is added to the hull $K_1$ at some time $T(z) < 1$. 
(This follows from Proposition 4.27 in \cite{lawler}, which says that there is at most one $t$-accessible point.)
Therefore as $s \to T(z)$, $f_s^{-1}(z) \to \lambda(T(z)) \geq a(1-T(z))^r$. 
 Since $p_0^{T(z)} \leq  \frac{8}{a}(1-T(z))^{1-r}$,
 there exists some $ T(z)<s < 1$ so that $\text{Re}(f_s^{-1}(z))   \in [p_0^s+\sqrt{1-s}, \,2\sqrt{1-s}] $.

Now suppose that $4 \leq a < M$.  Let $s \in (0, 1)$ satisfy that $a (1-s)^{r-1/2} = M$.
Since the driving function $\hat{\l}$ of  $\frac{1}{\sqrt{1-s}} \hat{K}_{s,1}$ satisfies $\hat{\l}(t) \geq M (1-t)^r$, the previous case implies that 
$$\left( \frac{1}{\sqrt{1-s}} \hat{K}_{s,1} \right) \cap \{x+iy \, : \, x \leq 2 \} \subset  \{x+iy \, : \, 0 \leq x, \, 0 \leq y \leq 52M^{-1}(x-\hat{p})^{2-2r}\}$$
for $\hat{p} = \inf \{ x \in \frac{1}{\sqrt{1-s}} \hat{K}_{s,1} \cap \mathbb{R} \}$.
The desired result follows since $f_{s}(\sqrt{1-s}\, z)$ is conformal in a neighborhood of $\hat{p}$ and takes $\frac{1}{\sqrt{1-s}} \hat{K}_{s,1}$ to $K_1\setminus K_s$ with  $\hat{p}$ mapping to $p$.

It remains to show that the constant $C$ in the statement of Theorem 2 only depends on $a$ and $r$.  
This will follow from showing that $|f_s'(p_0^s)|$ is bounded below, 
since when $y=Cx^b$, then $(\hat{x},\hat{y}) = (kx, ky)$ satisfy $\hat{y} = C k^{1-b}\hat{x}^b$.  
Note that by Schwarz reflection, $f_s$ can be extended to be conformal in $\mathbb{C} \setminus I$ for an interval $I \in \mathbb{R}$.  By the distortion theorem
$$|f'_s(p_0^s)| \geq \frac{\text{dist}(p, K_s)}{\text{dist}(p_0^s, I)}.$$
Set $d = \frac{1}{10} \text{min} \{ \xi_t \, : \, t \in [0,s]\}$.
Then $\text{dist}(p_0^s, I) \leq 10d$.   We claim that $\text{dist}(p, K_s) \geq d$.

To prove the claim we will show that $\text{dist}(p_t, \hat{K}_{s-t,s}) \geq d $  for all $t \in [0,s]$.
This holds at time 0, since $\hat{K}_{s,s} = \{\xi_0\}$.  
For $t$ close to 0, $\hat{K}_{s-t,s}$ is near $\xi_t$ and $\xi_t - p_t  \geq \xi_t/2 \geq 5d$.
Let $\tau$ be the first time $t$ when $\text{dist}(p_t, \hat{K}_{s-t,s}) = 2d $.
Let $z_\tau  \in \hat{K}_{s-\tau,s}$ with $| p_\tau - z_\tau | = 2d$, and for $t \in [\tau, s]$ let $z_t = x_t+iy_t$ satisfy \eqref{upLE}. 
 If $y_\tau > d$, then $| p_t - z_t | > d$ for all $t \in [\tau, s]$ since $z_t$ is moving upwards.  
 Suppose $y_\tau \leq d$, which means that $|x_\tau - p_\tau | \geq d$.  We will consider the case that $x_\tau - p_\tau  \geq d$, as the other case is similar.
Then by \eqref{xminusp} 
$$ \partial_t\left( x_t - p_t \right) \arrowvert_{t=\tau} 
= 2 \frac{(x_\tau - p_\tau)(\xi_\tau-p_\tau)-4d^2}{(\xi_\tau-p_\tau)[(\xi_\tau-x_\tau)^2+y_\tau^2]}.
$$
Note that the numerator satisfies
$$ (x_\tau - p_\tau)(\xi_\tau-p_\tau)-4d^2 \geq d \, \frac{\xi_\tau}{2} -4d^2 >0.$$
Therefore the distance between $z_t$ and $p_t$ is increasing at time $\tau$, which shows that $\text{dist}(p, K_s) \geq d$.

\end{proof}

\section{Discussion of trace existence and examples}\label{Tr}

In this section we  discuss  the existence of a trace curve, especially in the context of the driving functions that we are considering in this paper, i.e. those with end behavior bounded by $a(T-t)^r$ for $r \in (0,1/2)$.  We consider the following question:

\begin{question} \label{traceQ}
Let $\lambda : [0,T] \to \mathbb{R}$ be a continuous function such that the corresponding Loewner hull $K_t$ has a trace curve for $t \in [0,T)$.
What additional conditions are needed to guarantee that $K_t$ has a trace curve for $t \in [0,T]$?
\end{question}

Questions such as this about the existence of the trace have often proved to be difficult to answer.  
When $| \l(T) - \l(t) | / \sqrt{T-t} $ is bounded as $t \to T$, then Theorem 1.2 in \cite{ZZ} gives one possible answer to this question.  Since this result does not apply when the driving function is faster than $a(T-t)^r$ for $r \in (0,1/2)$ as $t\to T$, we are interested in other answers to Question \ref{traceQ}, such as the following result.

\begin{thmtrace}
Let $\l:[0,T] \to \mathbb{R}$ satisfy $\l \in C^2[0,T)$ and 
$ |\lambda(T)-\lambda(t)| \geq 4\sqrt{T-t}$ for all $t \in [0,T)$.
Assume the Loewner curvature satisfies $9 \leq LC_\lambda(t) < \infty$ for all $t \in [0,T)$,
and there exists $\delta>0$ so that for $s \in (0,1)$, 
$$\inf_{t \in [s,T)} LC_{\l}(t) \geq \delta \sqrt{T-s} \, \l'(s).$$
Then the trace $\g$ driven by $\l$ satisfies that $\displaystyle \gamma(T) = \lim_{t \to T} \gamma(t)$
exists and is real.
\end{thmtrace}

\begin{proof}
From \eqref{LCdef} and the bound on Loewner curvature, $\l'(t) \neq 0$ for all $t \in [0,T)$.   
Hence, $\l$ must be monotone and $\l'$ does not change sign. 
We make the following simplifying normalizations: 
by the scaling property, we may assume that $T=1$,  
 by the translation property, we may assume that $\l(1)=0$, 
and by the reflection property, we may assume that $\l'(t) < 0$.

If $t \in [0,1)$, then the  Loewner hull $K_t$ driven by $\l$ satisfies $K_t=\gamma[0,t]$ for a simple curve $\gamma$ in $\mathbb{H} \cup \{ \l(0) \}$, by Theorem \ref{cis4}.  
Lemma \ref{capture} 
guarantees that $K_1 \cap \mathbb{R}$  is a non-degenerate interval with right endpoint $\l(0)$.
Set $p=\inf \{ x \in K_1 \cap \mathbb{R} \}$ be the left endpoint.
We wish to show that 
$$ \lim_{t \nearrow 1} \gamma(t) = p.$$

First we will rule out the case that there are additional limit points of $\gamma$ in $\mathbb{R}$.
By way of contradiction, we assume that there is $q \in (p, \l(0)) \subset \mathbb{R}$ so that 
$\gamma(t_n) \to q $ for a sequence $t_n$ increasing to 1.
Let $q_t = g_t(q)$ be the solution to  \eqref{downLE}.
By Lemma \ref{ZZlem}, 
this implies that 
$$\liminf_{t \to 1} \frac{\l(t) - q_t}{\sqrt{1-t}}= 0.$$
Choose $s$  so that 
$$\frac{\l(s) - q_s}{\sqrt{1-s}} < \frac{\delta}{2}.$$
Consider the mapped and rescaled hulls 
$$\hat{K}_t = \frac{1}{\sqrt{1-s}} g_s(K_{s+t(1-s)} \setminus K_s),$$ 
which are generated by the driving function
$$\hat{\l}(t) = \frac{1}{\sqrt{1-s}}  \l(s+t(1-s)), \;\;\;   t \in [0,1].$$ 
Note that for $t <1$, $\hat{K}_t = \hat{\gamma}[0, t]$ for a simple curve $\hat{\gamma}$ and $\hat{q} =\frac{q_s}{\sqrt{1-s}} \in \hat{K}_1$ is a limit point of $\hat{\gamma}(t)$ as $t \to 1$ and satisfies that $\hat\l(0) - \hat q < \delta/2$.

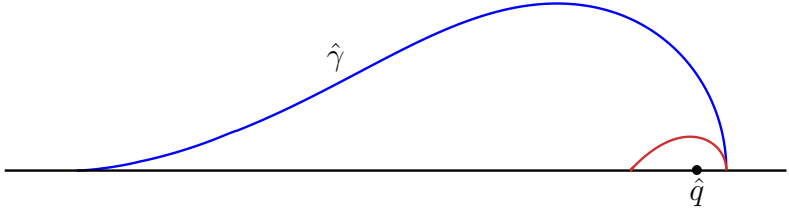
\begin{figure}
\centering
\begin{tikzpicture}
  \draw (0, 0) node[inner sep=0]
{\includegraphics[scale=.6]{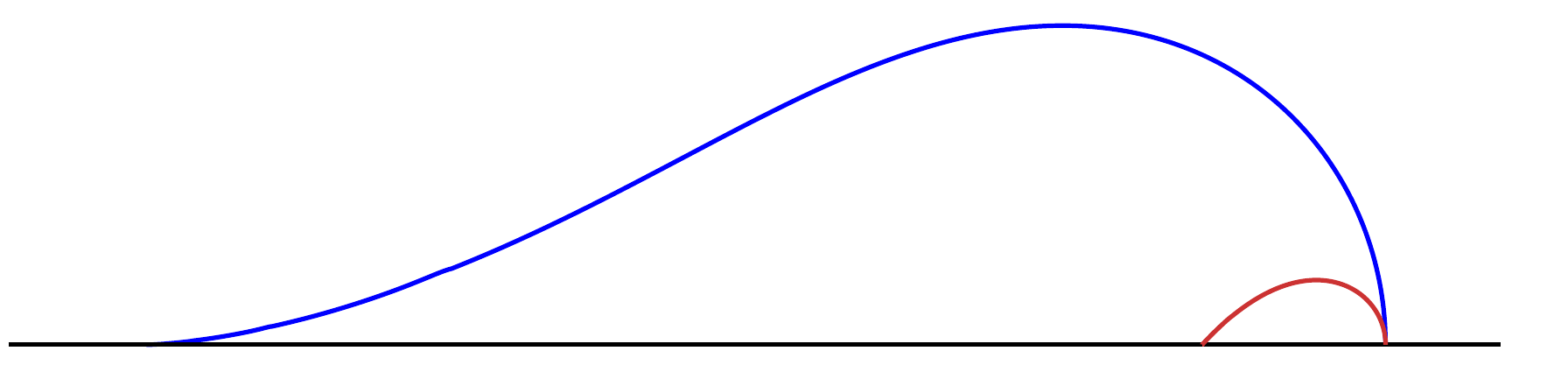}};
  \draw (3.8, -1.4) node {$\hat q$};
    \draw (-1, 0.4) node {$\hat \gamma$};
  \begin{scope}[every node/.style={circle,draw=black,fill=black!100!,font=\sffamily\Large\bfseries}]
   \node (v1) [scale=0.3] at (3.8,-1.1) {};
    \end{scope}
\end{tikzpicture}
\caption{The Loewner curvature comparison principle implies that $\hat{\gamma}$ (shown in blue) does not intersect the region below the smaller curve (shown in red) that contains $\hat{q}$.
} \label{LCfig}
\end{figure}

To obtain a contradiction, we will utilize the Loewner curvature comparison principle, 
which will show that there is a relatively open set $B$ in $\overline{\mathbb{H}}$
containing $\hat{q}$ so that
$B \cap \hat{\g}[0,T) = \emptyset$,  
as illustrated in Figure \ref{LCfig}.
We  compare $\hat{\l}$ to
 $\mu(t) = \alpha + c\sqrt{\tau-t}$,
 with the constants chosen as follows:
  $c$ is chosen so that $c^2/2= \inf_{t \in [s,T)} LC_{\l}(t)$,
   $\tau$  is chosen so that $\hat{\l}'(0)= \mu'(0)$, 
 and  $\alpha$  is chosen so that $\hat{\l}(0) = \mu(0)$.
Since  $$ LC_{\hat\l}(t) = LC_{\l}(s+t(1-s)) \geq \frac{c^2}{2} \geq 9, $$
 the Loewner curvature comparison principle (Theorem \ref{LCthm})  
implies that $\hat\gamma$ stays above and never intersects the interior of the hull driven by $\mu$.
It remains to show $\hat q$ is contained in this hull.
Note by Lemma \ref{capture}
the hull driven by $c\sqrt{1-t}$ contains a real interval of length at least $c/2$. 
  Hence by scaling, the hull driven by $\mu(t) = \alpha+ c\sqrt{\tau}  \sqrt{1-t/\tau}$ contains an interval of length at least
  $$\frac{c\sqrt{\tau}}{2} = \frac{c^2}{4\sqrt{1-s} \, \l'(s)} \geq \frac{\delta}{2}$$
  with right endpoint $\hat\l(0)$. 
 This  implies that
 $\hat{q}$ is contained in the interior of the hull generated by $\mu$ (using the relative topology of $\overline{\mathbb{H}}$), and hence $\hat{q}$ cannot be a limit point of $\hat\gamma$.

To finish the proof, it remains to show that there cannot be a limit point of $\gamma$ as $t \to 1$ in $\mathbb{H}$.  
If there were, the set limit points of $\gamma$ as $t \to 1$ would be a continuum 
containing $p \in \mathbb{R}$ and $\gamma$ would need to oscillate, such as in Figure \ref{oscillation}.  
Since the set of limit points extends into $\mathbb{H}$, 
 $\gamma$ must alternate between following its left side, i.e. the prime ends $g_t^{-1}(x)$ for $x < \l(t)$, and its right side.
This would require $\l$ to be non-monotone.

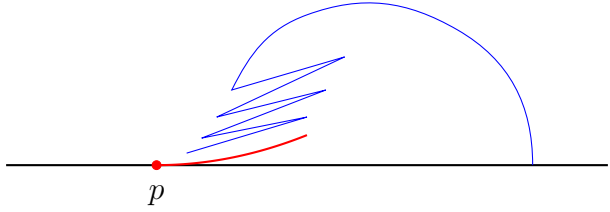
\begin{figure}
\centering
\begin{tikzpicture}
\draw[thick] (-2,0) -- (6,0);
\draw[blue] (5,0) to [out=90,in=-30] (4,1.8) to [out=150,in=20] (2,2) to [out=200,in=63.4] (1,1) to (2.5,1.44) to (2.5,1.44) to  (0.8,0.64) to (2.25,1) to (0.6,0.36) to (2.0,0.64) to  (0.4,0.16);
\draw [red, thick,  domain=0:2, samples=40] 
 plot ({\x}, {\x^2/10} );
  \begin{scope}[every node/.style={circle,draw=red,fill=red!100!,font=\sffamily\Large\bfseries}]
   \node (v1) [scale=0.3] at (0,0) {};
    \end{scope}
    \draw (0, -0.4) node {$p$};
 \end{tikzpicture}
\caption{The beginning of $\gamma$ (shown in blue) when the set of limit points of $\gamma$ as $t \to 1$ is a continuum  (shown in red) containing $p \in \mathbb{R}$ and points in $\mathbb{H}$.  
} \label{oscillation}
\end{figure}

\end{proof}

We now apply Proposition \ref{trace} and Theorem \ref{tangentialapproach} to analyze the Loewner hulls driven by $a(1-t)^r.$

\begin{proof}[Proof of Corollary \ref{rdrivers}] 
By scaling, we may assume that $T=1$, and by reflection we may assume that $a>0$.  Thus we take
 $\lambda(t) = a(1-t)^r$ for $a>0$ and $r \in (0, 1/2)$, and we let $K_t$ be the Loewner hulls driven by $\l$. 
By Theorems \ref{cis4} and \ref{smoothtrace},  there is a simple curve $\gamma  \in C^{2.5}(0,T)$ so that $K_t=\gamma[0,t]$ for $t \in [0,T)$.

We first assume that $a\geq \frac{3\sqrt{1-r}}{r}$.
Since $\frac{3\sqrt{1-r}}{r} \geq 3\sqrt{2} >4$,
this assumption  guarantees that $|\l(1) - \l(t)| \geq 4\sqrt{1-t}$.
Computing Loewner curvature gives
$$ LC_{\l}(t) = \frac{\l'(t)^3}{\l''(t)} =\frac{a^2r^2}{1-r} \cdot \frac{1}{(1-t)^{1-2r}} \geq 9 $$
and for $\delta = ar/(1-r)$,
$$\frac{\inf_{t \in [s,1)} LC_{\l}(t)}{\sqrt{1-s}\, \l'(s)} = \frac{ar}{1-r} \frac{1}{(1-s)^{1/2-r}} \geq \delta.$$ 
Thus Proposition \ref{trace} implies that $\displaystyle \gamma(T) = \lim_{t \to T} \gamma(t)$ exists and is real.
Theorem \ref{tangentialapproach} implies that $\g(t)$ approaches $\mathbb{R}$ tangentially as $t \to 1$.

The result for $a < \frac{3\sqrt{1-r}}{r} $ follows from the large $a$ case and the concatenation property.

\end{proof}

Since the monotonicity of the driving function played a role in the proof of Proposition \ref{trace}, it is natural to ask whether this property is sufficient to answer Question \ref{traceQ}.  The following example shows that monotonicty alone is not enough to guarantee the existence of a trace on the full time interval.
We also note that this example could be modified so that the driving function is in $C^2[0,1)$, showing that the problem is not the lack of smoothness.

\begin{prop}\label{NoTrace}
Let $r \in (0, 1/2)$.
There exists a continuous monotone driving function $\lambda : [0,1] \to \mathbb{R}$ 
with $\l(1)=0$ and $\l(t) \geq a (1-t)^r$,  where $a>0$, 
 such that the corresponding Loewner hull $K_t$ is a simple curve for $t \in [0,1)$, but $K_1$ does not have a trace.
 \end{prop}
 
 \begin{proof}
The driving function $\l$ will be constructed to alternate between constant and linear portions, as pictured in Figure \ref{NoTraceDriver}.
In particular, each interval of the form $I_n= [1-2^{-n}, 1-2^{-(n+1)}]$ is divided into two subintervals .  On the first subinterval,  $\l$ is constant, equal to $ 2^{-nr}a$, where $a$ satisfies $a \geq 2/(1-2^{-r})$.  On the second subinterval, $\l$ is linear.  
Since we require that $\l$ is continuous, choosing the slope of the linear piece will uniquely identify $\l$ on $I_n$.
For $t<1$, this construction will give a  simple curve $\gamma$ in $\mathbb{H} \cup \{a\}$ so that $K_t = \gamma[0,t]$. 
Let $\beta_0$ be the line segment $\{x+iy \, : \, x \in [a-1,a], y = a-x \}$. 
We will construct a nested sequence of subintervals $\beta_n$ converging to $a$ and we will choose the slopes of the linear portions to guarantee that $\gamma$ intersects each $\beta_n$.
This will show that the limit points of $\g$ as $t \to 1$ is an interval in $\mathbb{R}$.

\begin{figure}
\centering
\includegraphics[scale=.7]{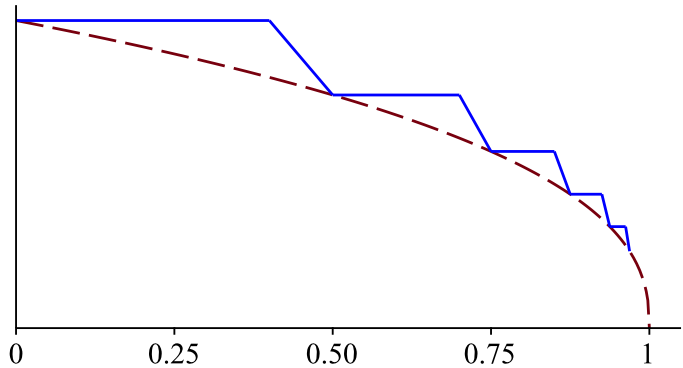}
\caption{The beginning of the driving function $\l$ of Proposition \ref{NoTrace} is shown in solid blue.  It is bounded by the dashed curve $a(1-t)^r$.
} \label{NoTraceDriver}
\end{figure}

We begin with the first interval $I_0$.  Let $s \in (0,1/2)$, and let $\l \equiv a$ on $[0,s]$, 
which implies $\gamma[0,s]$ is the vertical slit from $a$ to $a + i2\sqrt{s} $.  
Applying $g_s$,  the curve $g_s(\beta_0)$ has an endpoint at $a-2\sqrt{s}$. 
For $t \in [s, 1/2]$, we set $\l(t) = m_s(t-s)+a$, where $m_s = \frac{a(2^{-r}-1)}{1/2-s}$ and we let $\gamma_s$ be the Loewner curve generated by $\l$ restricted to $[s,1/2]$.  This curve begins at $a$ and moves to the left.  Making $s$ closer to $1/2$ increases the slope $m_s$, which in turn makes $\g_s[s,1/2]$ closer to $\mathbb{R}.$
As $s \to 1/2$, $\gamma_s[s,1/2]$ converges to the real interval $[2^{-r}a, a]$ of length $a(1-2^{-r})\geq 2$.  
Since the distance from $a$ to the endpoint of $g_s(\beta_0)$ is $2\sqrt{s}<\sqrt{2}$, we are able to choose $s$ close enough to 1/2, so that $\gamma_s[s,1/2]$ intersects $g_s(\beta_0)$.
This gives us our definition of $\l$ on $[0,1/2]$ and we set $\beta_1$ to be the connected component of $\beta_0 \setminus \gamma[0,1/2]$ containing $a$.
Note that we may assume that $g_{1/2}(\beta_1)$ is as close to $2^{-r}a$ as we like (by simply taking $s$ closer to $1/2$, if needed.)
The construction for subsequent intervals is similar.

\end{proof}

Despite the lack of trace, we note that Theorem \ref{tangentialapproach} still applies to the above example.
We end this section by discussing two further examples where we can apply Theorem \ref{tangentialapproach} but which lack the regularity of Proposition \ref{trace}.  It is currently unknown whether either has a trace curve.

The first example behaves similarly to the driving function of Proposition \ref{NoTrace} in that it is monotone and has periods where it is constant.  In particular, we are interested in the driving function $k\sqrt{1-E_t}$, where $E_t$ is an inverse $\alpha$-stable subordinator.  See Figure \ref{sqrtTimeChange}.
In \cite{KLS} with Kobayashi and Starnes we looked at random time-changed driving functions of the form $\phi(E_t)$, and as an application of our results,  we showed that when $\alpha > 1/2$, then  a.s.~$k\sqrt{E_t}$ generates a trace curve that leaves the real line tangentially.  
Analyzing $k\sqrt{1-E_t}$ is more difficult.  When $\alpha > 1/2$, the work of \cite{KLS} shows that a.s.  $k\sqrt{1-E_t}$ generates a trace curve on $[0,T)$, before the final time $T$.  When the hull includes points from the real line, then Theorem \ref{tangentialapproach} gives the tangential behavior of the final hull.  However, the question remains open whether the trace exists on the full time interval $[0,T]$.

\begin{figure}
\centering
\includegraphics[scale=.7]{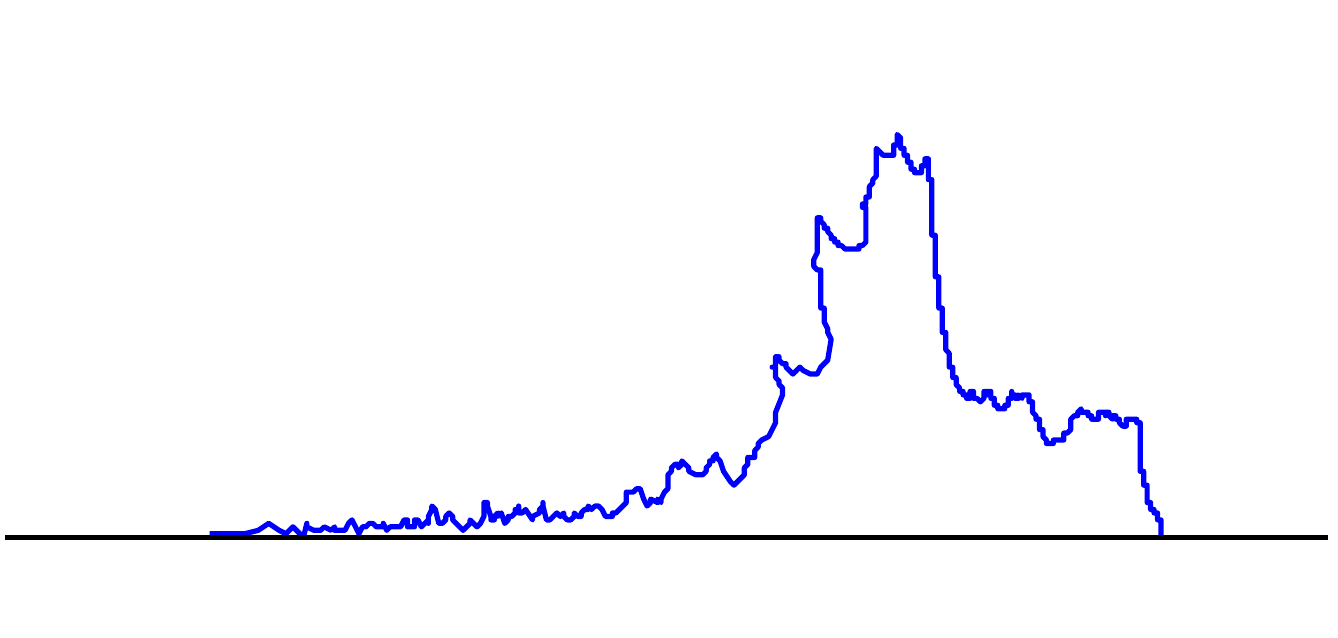}
\caption{
A simulation, courtesy of Andrew Starnes, of the Loewner trace driven by
 $\sqrt{1-E_t}$, where $E_t$ is an inverse $\alpha$-stable subordinator with $\alpha = 0.7$.
} \label{sqrtTimeChange}
\end{figure}

The second example comes from the family of Weierstrass functions
$$W(t) = W_{b,r,k}(t) = k \sum_{n=0}^\infty b^{-rn} \cos(b^n t).$$
The Loewner hulls driven by $W$ have been studied in the $r=1/2$ case (see \cite{LR}, \cite{G}, \cite{ZZ}.)  When $r \in (0,1/2)$ (and $k$ large enough), then we enter the situation in which Theorem \ref{tangentialapproach} applies.  A simulation of one such example is shown in Figure \ref{Wthirdroot}.
The tangential behavior on the left side of the hull is due to Theorem \ref{tangentialapproach}, whereas the tangential behavior on the right side of the hull is due to Proposition 1.2 of \cite{KLS}.  
We also note that since the simulation that produced Figure \ref{Wthirdroot} creates a trace that approximates the hull,  this picture suggests that the hull may be a spacefilling curve (and the few white spots are most likely approximation error),
but it is unknown whether the trace exists for this example.

\vspace{0.1in}

\noindent {\bf Acknowledgement:}  We thank Andrew Starnes for his comments.

%

\end{document}